\documentclass[11pt]{amsart}
\usepackage{amssymb}
\usepackage{bm}
\usepackage[centertags]{amsmath}
\usepackage{amsfonts}
\usepackage{amsthm}
\usepackage{mathrsfs}
\usepackage[normalem]{ulem}
\usepackage{comment}
\usepackage{enumitem}
\usepackage{mathabx}
\usepackage{hyperref}
\usepackage[noabbrev]{cleveref}
\usepackage{cleveref}

\usepackage{amsthm}

\newtheorem*{theorem*}{Theorem}
\usepackage{xifthen}
\usepackage{ifthenx}
\usepackage{xargs}

\usepackage[table,svgnames,x11names]{xcolor}
\usepackage[colorinlistoftodos,prependcaption,textsize=tiny,textwidth=3.5cm]{todonotes}
\presetkeys{todonotes}{fancyline}{}
\newcommandx{\change}[2][1=]{\todo[linecolor=blue,backgroundcolor=blue!25,bordercolor=blue,#1]{#2}}
\newcommandx{\changein}[2][1=]{\change[inline, caption={change}, #1]{%
    \begin{minipage}{\textwidth-20pt}#2\end{minipage}}}
\newcommandx{\todoin}[2][1=]{\todo[inline, caption={todo}, #1]{%
    \begin{minipage}{\textwidth-20pt}#2\end{minipage}}}

\newcommandx{\remove}[2][1=]{\todo[linecolor=Plum,backgroundcolor=Plum!25,bordercolor=Plum,#1]{#2}}
\newcommandx{\removein}[2][1=]{\remove[inline, caption={todo}, #1]{%
    \begin{minipage}{\textwidth-20pt}#2\end{minipage}}}

\makeatletter
\def\l@subsection{\@tocline{2}{0pt}{2.5pc}{2.5pc}{}}%
\makeatother
\makeatletter
\def\l@subsubsection{\@tocline{3}{0pt}{5pc}{5pc}{}}%
\makeatother

\usepackage{hyperref}
\hypersetup{pdftex,colorlinks=true,allcolors=black}
\usepackage{hypcap}

\DeclareFontFamily{OT1}{pzc}{}
\DeclareFontShape{OT1}{pzc}{m}{it}{<-> s * [1.10] pzcmi7t}{}
\DeclareMathAlphabet{\mathpzc}{OT1}{pzc}{m}{it}

\theoremstyle{definition}
\newtheorem{thm}{Theorem}[section]
\newtheorem{dfn}[thm]{Definition}

\newtheorem{prop}[thm]{Proposition}

\newtheorem{rem}[thm]{Remark}

\newtheorem{ntt}[thm]{Notation}

\newtheorem*{defn*}{Definition}

\newtheorem*{thm*}{Theorem}

\newtheorem*{cor*}{Corollary}

\newtheorem*{prp*}{Proposition}

\newcommand{\trn}[1]{{\left\vert\kern-0.25ex\left\vert\kern-0.25ex\left\vert #1 
    \right\vert\kern-0.25ex\right\vert\kern-0.25ex\right\vert}}

\newcommand{\trnsmall}[1]{{\vert\kern-0.25ex\vert\kern-0.25ex\vert #1 
    \vert\kern-0.25ex\vert\kern-0.25ex\vert}}

\long\def\symbolfootnote[#1]#2{\begingroup%
\def\thefootnote{\fnsymbol{footnote}}\footnote[#1]{#2}\endgroup}

\makeatletter
\@namedef{subjclassname@2020}{%
  \textup{2020} Mathematics Subject Classification}
\makeatother

\allowdisplaybreaks

\begin{document}

\title[On the SHAI property of the $R_\alpha^p$ spaces]{On the SHAI property of the Bourgain-Rosenthal-Schechtman spaces}

\author[K. Konstantos]{Konstantinos Konstantos}

\address{Konstantinos Konstantos, Department of Mathematics and Statistics, York University, 4700 Keele Street, Toronto, Ontario, M3J 1P3, Canada}

\email{konstantoskostas95@gmail.com}



\begin{abstract}

A Banach space $X$ has the SHAI property if, for every non-zero Banach space $Y$, every surjective algebra homomorphism from the algebra $\mathcal{L}(X)$ of bounded linear operators on $X$ onto $\mathcal{L}(Y)$ is injective. In this work, we prove that the Bourgain-Rosenthal-Schechtman spaces have the SHAI property, thereby answering a question posed by Johnson, Phillips, and Schechtman in 2022.

\end{abstract}

\maketitle

\tableofcontents

\section{Introduction} \label{introd}

In \cite{horvath:2020}, Horvath introduced the concept of the SHAI property (Surjective Homomorphisms Are Injective). A Banach space $X$ is said to have the SHAI property if, for every non-zero Banach space $Y$, every surjective algebra homomorphism from $\mathcal{L}(X)$ onto $\mathcal{L}(Y)$ is injective. In this case, the Banach spaces $X$ and $Y$ are necessarily isomorphic by a classical theorem of Eidelheit \cite{eidelheit:1940}. The first elementary examples of Banach spaces with the SHAI property are the finite-dimensional Banach spaces; see the second paragraph of Section~1 in \cite{horvath:2020}. In the same paper, Horvath established the SHAI property for several classical Banach spaces, including $c_0$, $\ell_p$ for $1\leq p\leq\infty$, and Hilbert spaces of arbitrary density. He also exhibited several Banach spaces that fail to have the SHAI property, including the James space, $C[0,\omega_1]$, and the complex hereditarily indecomposable Banach spaces.

In \cite{horvath:kania:shai:2021}, Horvath and Kania established the SHAI property for the long sequence spaces $c_0(\lambda)$, $\ell_{\infty}^{c}(\lambda)$, and $\ell_p(\lambda)$ for every infinite cardinal $\lambda$ and every $p\in[1,\infty)$. In \cite{johnson:phillips:schechtman:2022}, Johnson, Phillips, and Schechtman proved that spaces with subsymmetric bases and finite cotype, as well as the Schatten $p$-spaces for $1<p<\infty$, have the SHAI property. They also investigated the SHAI property for complemented subspaces of $L_p$, $1<p<\infty$. In particular, they proved that $L_p$, $\ell_p\oplus\ell_2$, $\ell_p(\ell_2)$, the Rosenthal space $X_p$ \cite{rosenthal:1970:Xp}, and the Schechtman spaces $X_p^{\otimes n}$, $n\in\mathbb{N}$ \cite{schechtman:1975}, have the SHAI property.

The present work is motivated by a question posed by Johnson, Phillips, and Schechtman in \cite{johnson:phillips:schechtman:2022}, asking whether the Bourgain--Rosenthal--Schechtman spaces $R_{\alpha}^{p}$ have the SHAI property. We answer this question affirmatively. In the same paper, they asked whether every complemented subspace of $L_p$, or, more specifically, every complemented subspace of $L_p$ with an unconditional basis, has the SHAI property. Both questions remain open.

We focus on the SHAI property of the Bourgain--Rosenthal--Schechtman spaces. In their seminal 1981 paper \cite{bourgain:rosenthal:schechtman:1981}, see also \cite{alspach:1999, konstantos:motakis:2025, konstantos:motakis:brs:2025, konstantos:2026}, Bourgain, Rosenthal, and Schechtman introduced a transfinite family of complemented subspaces of $L_p$, denoted by $R_\alpha^p$, $0\leq\alpha<\omega_1$. They used this family to establish the existence of uncountably many pairwise non-isomorphic complemented subspaces of $L_p$.

Our main result is the following.

\begin{theorem*}
Let $\alpha<\omega_1$. For $1<p<\infty$, the Bourgain--Rosenthal--Schechtman space $R_\alpha^p$ has the SHAI property.
\end{theorem*}

The paper is organized as follows. In Section \ref{not and con}, we recall a sufficient condition, established in \cite{johnson:phillips:schechtman:2022}, for a Banach space with an unconditional FDD to have this property. Specifically, it was shown that if the FDD satisfies property $(\#)$—namely, there exists an almost disjoint continuum $\lbrace N_{\gamma}: \gamma < \mathfrak{c} \rbrace$ of infinite subsets of $\mathbb{N}$ such that, for every $\gamma< \mathfrak{c}$, the space $X$ is isomorphic to the closed linear span of the FDD components indexed by $N_{\gamma}$—and the FDD satisfies a disjoint lower $p$ estimate for some $p<\infty$, then $X$ has the SHAI property. In this section, we prove that the aforementioned sufficient condition remains valid if property $(\#)$ is replaced by property $(\#^{\prime})$, in which the space $X$ is isomorphic to a complemented subspace of the closed linear span of the FDD components indexed by $N_{\gamma}$. Later, in Section \ref{section for SHAI of BRS}, we will use this variation of the sufficient condition to establish the SHAI property of the spaces $R_{\alpha}^{p}$.

We conclude Section \ref{not and con} by recalling the standard Haar system notation used in the recursive definition of the spaces $R_{\alpha}^{p}$ in Section \ref{recal BRS spaces}. In Section \ref{recal BRS spaces}, we recall a distributional representation of the Bourgain-Rosenthal-Schechtman spaces, first introduced in \cite{konstantos:motakis:brs:2025}, which allows us to view these spaces as subspaces of $L_{p}[0,1]$. Moreover, for every $\alpha < \omega_{1}$, we recall the structure of the isomorphic hyperplane $R_{\alpha}^{p,0}$ of $R_{\alpha}^{p}$ consisting of all random variables of mean zero.

In Section 4, we recall the unconditional FDD of the spaces $R_{\alpha}^{p,0}$, constructed in \cite{konstantos:motakis:brs:2025}. For every ordinal $\alpha < \omega_{1}$, there exists a countable well-founded tree $\mathcal{T}_{\alpha}$ and a collection of spaces $(X_\lambda)_{\lambda \in \mathcal{T}_{\alpha}}$, where each $X_\lambda$ is a compressed copy of a finite dimensional $L_p$ space, which together form an unconditional FDD of $R_{\alpha}^{p,0}$.

Following the discussion of the unconditional FDD of the $R_{\alpha}^{p,0}$ spaces, in Section 5 we focus on the main result of this work, namely, the SHAI property of the infinite-dimensional $R_{\alpha}^{p}$ spaces. More precisely, we establish the SHAI property of the $R_{\alpha}^{p,0}$ spaces, which immediately yields the SHAI property of the $R_{\alpha}^{p}$ spaces. The proof is based on the aforementioned variation of the sufficient condition discussed in Section 2. For every $\alpha < \omega_{1}$, the problem is reduced to the construction of an almost disjoint continuum $\lbrace M_\gamma (\alpha): \gamma < \mathfrak{c} \rbrace$ of infinite subsets of $\mathcal{T}_{\alpha}$ such that, for every $\gamma < \mathfrak{c}$, the space $R_{\alpha}^{p,0}$ is isomorphic to a complemented subspace of the closed linear span of the FDD components indexed by $M_{\gamma} (\alpha)$. Indeed, $R_{\alpha}^{p,0}$ distributionally embeds into the aforementioned closed linear span, and therefore this distributional copy is complemented. The construction of the almost disjoint continuum is recursive over the countable ordinals, and for the purpose of obtaining a distributional representation of $R_{\alpha}^{p,0}$ into the aforementioned closed linear span, we assemble appropriate distributional embeddings.

\section{Notation and concepts} \label{not and con}

In this section, we introduce the standard terminology and notation used throughout the paper. We begin by recalling the necessary notions from probability theory and Banach space theory. We then review the concept of the SHAI property and conclude by recalling the notation for the Haar system.

We write \(\mathbb{N}=\{1,2,\dots\}\) for the set of positive integers, \(\mathbb{N}_{0}=\{0\}\cup\mathbb{N}\) and \(\mathbb{R}\) for the set of real numbers. Throughout the paper, we fix \(1<p<\infty\) and consider the space \(L_{p}([0,1])\) with respect to the Borel \(\sigma\)-algebra on \([0,1]\), denoted by \(\mathcal{B}([0,1])\), and the Lebesgue measure, denoted by \(|\cdot|\). Henceforth, we simply write \(L_{p}\) instead of \(L_{p}([0,1])\).

For a set \(A\subseteq[0,1]\), we denote by \(\chi_A:[0,1]\to\{0,1\}\) its characteristic function. If \(f\) is a real-valued random variable on \([0,1]\), then the distribution of \(f\), denoted by \(\mathrm{dist}(f)\), is the probability measure on \(\mathcal{B}(\mathbb{R})\) defined by 
\begin{align*}
(\mathrm{dist}(f))(A)=|[f \in A]|,\;\; A\in\mathcal{B}(\mathbb{R}).
\end{align*}
The characteristic function of \(f\) is defined by 
\begin{align*}
\Phi_f(t)=\int_{0}^{1}e^{itf(x)}\,dx, \;\; t\in\mathbb{R}.
\end{align*}
A fundamental theorem in probability theory states that, for two random variables \(f\) and \(g\), $\mathrm{dist}(f)=\mathrm{dist}(g)$ if and only if $\Phi_f(t)=\Phi_g(t)$, for all $t\in\mathbb{R}$. If \(f\) is a real-valued random variable on \([0,1]\), we denote its expectation by $\mathbb{E}(f)$.

For a subset $S$ of a linear (respectively, normed) space, we denote by $\langle S \rangle$ (respectively, $[S]$) its linear (respectively, closed linear) span. A family $(e_{a})_{a \in A}$ in a Banach space $X$ is called an unconditional Schauder basis of $X$ if every vector $x\in X$ admits a unique representation $x=\sum_{a\in A} \lambda_a e_{a}$, where $(\lambda_{a})_{a \in A}$ is a family of real numbers, and the convergence is unconditional. A family $(X_a)_{a\in A}$ of closed subspaces of a Banach space $X$ is called an unconditional Schauder decomposition of $X$ if every vector $x\in X$ admits a unique representation $x=\sum_{a\in A}x_a$, where $x_a\in X_a$ for every $a\in A$, and the convergence is unconditional. Every unconditional Schauder decomposition $(X_a)_{a\in A}$ of a Banach space $X$ induces a family of bounded linear projections $(P_a)_{a\in A}$ on $X$ defined by $P_b(\sum_{a\in A}x_a)=x_b$. If each $X_a$ is finite-dimensional, the family $(X_a)_{a\in A}$ is called an unconditional finite-dimensional decomposition (FDD). If $(X_{a})_{a \in A}$ is an unconditional Schauder decomposition of a Banach space $X$, then for every subset $B\subset A$, the net $\{\sum_{a\in F}P_a:F\subset B,\ F\ \text{finite}\}$ is bounded in $\mathcal{L}(X)$ and converges strongly to a projection $P_B$ onto $[X_a:a\in B]$. The suppression constant of the unconditional Schauder decomposition $(X_a)_{a\in A}$ is defined by $\sup\left\{\left\|\sum_{a\in F}P_a\right\|:F\subset A,\ F\ \text{finite}\right\}$. Consequently, for every subset $B\subset A$, the projection $P_B$ satisfies $\|P_B\|\leq\sup\left\{\left\|\sum_{a\in F}P_a\right\|:F\subset A,\ F\ \text{finite}\right\}$. The above topics from Banach space theory can be found in \cite{albiac:kalton:2006} and \cite{lindenstrauss:tzafriri:1977}.

\subsection{The SHAI property} Following the definition of the SHAI property given in the introduction, we recall a sufficient condition, established in \cite{johnson:phillips:schechtman:2022}, under which a Banach space with an unconditional FDD has this property.

In \cite{johnson:phillips:schechtman:2022}, the authors showed that if the FDD satisfies property $(\#)$---namely, that $X$ is isomorphic to the closed linear span of the FDD components indexed by each member of an almost disjoint continuum of infinite subsets of $\mathbb{N}$---and the FDD satisfies a disjoint lower $p$ estimate for some $p<\infty$, then $X$ has the SHAI property.

Next, we recall the concepts of property $(\#)$ and the disjoint lower $p$ estimate. Recall that a family of sets is called almost disjoint if the intersection of any two distinct members of the family is finite.

\begin{dfn}
Let $X$ be a Banach space. Suppose that $(X_{n})_{n=1}^{\infty}$ is an unconditional FDD of the space $X$. We say that $(X_{n})_{n=1}^{\infty}$ has property $(\#)$ provided there is an almost disjoint continuum $\lbrace N_{\gamma}: \gamma < \mathfrak{c} \rbrace$ of infinite subsets of $\mathbb{N}$ such that for each $\gamma < \mathfrak{c}$, the space $X$ is isomorphic to the space
\begin{align*}
 \Big[ \bigcup_{n \in N_\gamma} X_{n}  \Big].   
\end{align*}
\end{dfn}

The definition of the disjoint lower $p$ estimate is stated for unconditional Schauder decompositions.

\begin{dfn}
Let $X$ be a Banach space. Suppose that $(X_{a})_{a \in A}$ is an unconditional Schauder decomposition of $X$. We say that $(X_{a})_{a \in A}$ has a disjoint lower $p$ estimate for some $p < \infty$ provided that there is $C< \infty$ so that whenever $x_1,...,x_n$ are finitely many vectors in $X$ such that for every $a \in A$ there is at most one $i$ with $1 \leq i \leq n$ for which $P_\alpha x_i \neq 0$, we have the inequality 
\begin{align*}
\Big\| \sum_{i=1}^{n} x_i  \Big \| \geq  \frac{1}{C} \Big( \sum_{i=1}^{n} \| x_i \|^{p} \Big)^{1/p}.    
\end{align*}
The constant $C$ is called the constant of the disjoint lower $p$ estimate. 
\end{dfn}

In the following remark, we describe the relationship between the notions of a disjoint lower estimate and cotype. Recall that a Banach space $X$ is said to have finite cotype $q$ for some $2 \leq q <  \infty$ (see \cite[Definition 6.2.10]{albiac:kalton:2006}) if there is a constant $C < \infty$ such that for every finite sequence of vectors $x_1,...,x_n$ in $X$ , we have the inequality 
\begin{align*}
\Big( \mathbb{E}  \Big\| \sum_{k=1}^{n} \epsilon_k x_k \Big \|^{q} \Big)^{1/q} \geq \frac{1}{C} \Big( \sum_{k=1}^{n} \| x_k \|^{q} \Big)^{1/q},   
\end{align*}
where the expectation is taken with respect to the uniform probability measure on all sequences $(\epsilon_k)_{k=1}^n \in \{-1,1\}^n$.

\begin{rem} \label{cotype and dis lower estimate}
Note that if a Banach space $X$ has finite cotype $q$, then every unconditional Schauder decomposition of $X$ has disjoint lower $q$ estimate, where the constant depends only on the suppression constant of the decomposition and the cotype $q$ constant of $X$.
\end{rem}

Recall that \(L_p\) has finite cotype. Since \(R_{\alpha}^{p,0}\) is a subspace of \(L_p\), it follows that \(R_{\alpha}^{p,0}\) also has finite cotype. Hence, for every \(\alpha < \omega_1\), the unconditional FDD \((X_{\lambda})_{\lambda \in \mathcal{T}_{\alpha}}\) of \(R_{\alpha}^{p,0}\) (see Section \ref{Section for S-decompostion}) satisfies a finite disjoint lower estimate.

The proof of the following theorem can be found in \cite[Theorem 1.4]{johnson:phillips:schechtman:2022}.

\begin{thm} \label{suf cond for SHAI}
Let $X$ be a Banach space. Suppose that $(X_{n})_{n=1}^{\infty}$ is an unconditional FDD of $X$. Assume that $(X_{n})_{n=1}^{\infty}$ has property $(\#)$ and $(X_{n})_{n=1}^{\infty}$ has a disjoint lower $p$ estimate for some $p < \infty$. Then $X$ has the SHAI property. 
\end{thm}

For the purposes of this paper, we work with a variation of property $(\#)$, which we denote by $(\#')$.

\begin{dfn}
Let $X$ be a Banach space. Suppose that $(X_{n})_{n=1}^{\infty}$ is an unconditional FDD of $X$. We say that $(X_{n})_{n=1}^{\infty}$ has property $(\#^{\prime})$ provided there is an almost disjoint continuum $\lbrace N_{\gamma}: \gamma < \mathfrak{c} \rbrace$ of infinite subsets of $\mathbb{N}$ such that for each $\gamma < \mathfrak{c}$, the space $X$ is isomorphic to a complemented subspace of the space
\begin{align*}
 \Big[ \bigcup_{n \in N_\gamma} X_{n}  \Big].   
\end{align*}
\end{dfn}

Note that, in the above definition, we regard each \(N_{\gamma}\) as a subset of \(\mathbb{N}\), serving as an index set. In Section~\ref{section for SHAI of BRS}, we study property \((\#')\) for the infinite-dimensional spaces \(R_{\alpha}^{p,0}\). Since each space \(R_{\alpha}^{p,0}\) admits an unconditional FDD indexed by a countable well-founded tree \(\mathcal{T}_{\alpha}\), the corresponding index sets \(N_{\gamma}\) will be constructed as subsets of \(\mathcal{T}_{\alpha}\).

For the remainder of this subsection, we prove that Theorem~\ref{suf cond for SHAI} remains valid when property $(\#)$ is replaced by property $(\#')$. We state the corresponding variant of the aforementioned theorem.

\begin{thm} \label{suf cond of SHAI property, var}
Let $X$ be a Banach space. Suppose that $(X_{n})_{n=1}^{\infty}$ is an unconditional FDD of $X$. Assume that $(X_{n})_{n=1}^{\infty}$ has property $(\#^{\prime})$ and $(X_{n})_{n=1}^{\infty}$ has a disjoint lower $p$ estimate for some $p < \infty$. Then $X$ has the SHAI property. 
\end{thm}

We prove Theorem~\ref{suf cond of SHAI property, var}. As in the proof of Theorem~\ref{suf cond for SHAI} (see the proof of \cite[Theorem~1.4]{johnson:phillips:schechtman:2022}), the only point at which the isomorphism appearing in the definition of property $(\#)$,
\[
X \cong \Big[ \bigcup_{n \in N_{\gamma}} X_n \Big], \qquad \gamma < \mathfrak{c},
\]
is used is in proving that, whenever $\Phi : \mathcal{L}(X) \to \mathcal{A}$ is a nonzero homomorphism onto a Banach algebra $\mathcal{A}$, the element $\Phi(P_{N_\gamma})$ is a nonzero idempotent in $\mathcal{A}$ for every $\gamma < \mathfrak{c}$ ( an element $p \in \mathcal{A}$ is called an idempotent if $p^{2}=p$ ). This is precisely the content of \cite[Proposition~1.3]{johnson:phillips:schechtman:2022}, which is subsequently used in the proof of \cite[Theorem~1.4]{johnson:phillips:schechtman:2022}. The remainder of the proof of \cite[Theorem~1.4]{johnson:phillips:schechtman:2022} does not rely on the above isomorphism from property $(\#)$. Therefore, once we establish that property $(\#')$ likewise implies that $\Phi(P_{N_\gamma})$ is a nonzero idempotent in $\mathcal{A}$ for every $\gamma < \mathfrak{c}$, Theorem~\ref{suf cond of SHAI property, var} follows immediately, since the rest of the argument is identical to the proof of Theorem~\ref{suf cond for SHAI} (see the proof of \cite[Theorem~1.4]{johnson:phillips:schechtman:2022}).

Therefore, the proof of Theorem \ref{suf cond of SHAI property, var} is reduced to establishing the following proposition.

\begin{prop} Let $X$ be a Banach space. Suppose that $(X_{n})_{n=1}^{\infty}$ is an unconditional FDD of the space $X$. Assume that $(X_{n})_{n=1}^{\infty}$ has property $(\#^{\prime})$, witnessed by an almost disjoint family $\lbrace N_{\gamma}: \gamma < \mathfrak{c} \rbrace$ of infinite subsets of $\mathbb{N}$. Suppose that $\Phi : \mathcal{L}(X) \to \mathcal{A}$ is a nonzero homomorphism onto a Banach algebra $\mathcal{A}$. Then, for every \(\gamma<\mathfrak{c}\), \(\Phi(P_{N_\gamma})\) is a nonzero idempotent in $\mathcal{A}$.
\end{prop}
\begin{proof}
We prove only that \(\Phi(P_{N_\gamma})\) is nonzero, since the fact that \(\Phi(P_{N_\gamma})\) is idempotent follows immediately from the multiplicativity of the algebra homomorphism \(\Phi\) and the identity $P_{N_\gamma}^2=P_{N_\gamma}$. Indeed, from property $(\#^{\prime})$ we have that $X$ is isomorphic to a complemented subspace of  
\[
\Big[ \bigcup_{n \in N_{\gamma}} X_n \Big].
\]
Denote $Z$ the complemented isomorphic copy of $X$ in the above closed linear span, and denote $T: X \to Z$ the isomorphism and $Q : X \to Z$ the projection onto $Z$. Let $S = T^{-1}QP_{N_\gamma} : X \to X$. Note that $SP_{N_\gamma}T  = \mathrm{id}_{X}$, where  $\mathrm{id}_{X}$ is the identity operator on $X$. We have that $\Phi(S) \Phi(P_{N_{\gamma}}) \Phi(T) = \Phi(\mathrm{id}_{X})$ and since $\Phi$ is surjective we obtain that $\Phi(\mathrm{id}_{X})$ is the identity element of $\mathcal{A}$. Hence, $\Phi(P_{N_{\gamma}})$ cannot be the zero element of $\mathcal{A}$.
\end{proof}

\subsection{Haar system notation}

In this subsection, we recall the standard notation for dyadic intervals, the Haar system, and finite-dimensional $L_p$ spaces.

\begin{ntt}
A dyadic interval of $[0,1]$ is a subinterval of $[0,1]$ of the form 
\begin{align*}
\Big[ \frac{i-1}{2^n}, \frac{i}{2^n} \Big),     
\end{align*}
where $n \in \mathbb{N}_0$ and $1 \leq i \leq 2^n$. We denote by $\mathcal{D}$ the family of all dyadic intervals of $[0,1]$. If $I \in \mathcal{D}$, we denote by $I^{+}$ and $I^{-}$ the left and right half of $I$ respectively, both are elements of $\mathcal{D}$, that is if 
\begin{align*}
I = \Big[ \frac{i-1}{2^n}, \frac{i}{2^n} \Big),     
\end{align*}
for some $n \in \mathbb{N}_0$ and $1 \leq i \leq 2^n$, then
\begin{align*}
I^{+} = \Big[ \frac{2(i-1)}{2^{n+1}}, \frac{2i-1}{2^{n+1}} \Big)\; \text{and}\; I^{-} = \Big[ \frac{2i-1}{2^{n+1}}, \frac{2i}{2^{n+1}} \Big).   
\end{align*}
For $n \in \mathbb{N}_0$, denote by $\mathcal{D}^{n}$ the set of all dyadic intervals of $\mathcal{D}$ with length $1/2^n$, that is 
\begin{align*}
\mathcal{D}^n = \Big \lbrace  \Big[ \frac{i-1}{2^n}, \frac{i}{2^n} \Big): 1 \leq i \leq 2^n  \Big\rbrace,    
\end{align*}
and by $\mathcal{D}_n$ the set of all dyadic intervals of $\mathcal{D}$ with length greater than or equal to $1/2^n$, that is
\begin{align*}
\mathcal{D}_n = \bigcup_{k=0}^{n} \mathcal{D}^{k}. 
\end{align*}
\end{ntt}

\begin{ntt}
For every $I \in \mathcal{D}$, denote 
\begin{align*}
h_I = \chi_{I^{+}} - \chi_{I^{-}}.   
\end{align*}
The standard Haar system of $L_p$ is the family $(\chi_{[0,1]}) \cup(h_{I})_{I \in \mathcal{D}
}$. The aforementioned family forms an unconditional basis of $L_p$ by the Paley–Burkholder inequality, and thus, 
\begin{align*}
L_p = \Big[(\chi_{[0,1]}) \cup(h_{I})_{I \in \mathcal{D}
}  \Big].    
\end{align*}
Denote by $L_{p}^{0}$ the subspace of $L_p$ consisting of all mean-zero random variables, that is
\begin{align*}
L_{p}^{0} = \Big[ h_{I}: I \in \mathcal{D} \Big].    
\end{align*}
With regard to the finite-dimensional $L_p$ spaces, for every $n \in \mathbb{N}$, we denote
\begin{align*}
L_{p}^{n}  = \Big< \chi_{I}: I \in \mathcal{D}^{n} \Big>  = \Big< (\chi_{[0,1]}) \cup (h_I)_{I \in \mathcal{D}_{n-1}} \Big>   
\end{align*}
and denote by $L_p^{n,0}$ the subspace of $L_p^n$ consisting of all random variables with expectation zero; that is,
\begin{align*}
L_{p}^{n,0}  = \Big< h_I : I \in \mathcal{D}_{n-1}\Big>.   
\end{align*}
\end{ntt}

\section{The Bourgain-Rosenthal-Schechtman  $R_{\alpha}^{p}$ spaces } \label{recal BRS spaces}

In this section, we recall the construction of the Bourgain-Rosenthal-Schechtman spaces $R_{\alpha}^{p}$. Originally introduced in \cite{bourgain:rosenthal:schechtman:1981}, these spaces are realized as subspaces of $L_p(\lbrace 0,1 \rbrace^{T_\alpha})$, where $T_\alpha$ is a countable index set associated with the ordinal $\alpha$. The spaces $R_{\alpha}^{p}$ are defined recursively over the countable ordinals. The initial space $R_{0}^{p}$ consists of the constant functions. At successor ordinals, the construction is based on disjoint sums, whereas at limit ordinals, the corresponding spaces are obtained through independent sums.

For the purposes of this paper, we work with an equivalent construction of the spaces $R_{\alpha}^{p}$ as subspaces of $L_p=L_p[0,1]$, introduced in \cite{konstantos:motakis:brs:2025}. In this construction, disjoint sums are represented by compressions with disjointly supported ranges, while independent sums are implemented by distributional embeddings with independent ranges. We will also consider the subspaces $R_{\alpha}^{p,0}$ consisting of those random variables in $R_{\alpha}^{p}$ having zero expectation. Finally, we recall the result from \cite{konstantos:motakis:brs:2025} that every distributional copy of $R_{\alpha}^{p,0}$ is complemented in $L_p$.

The operators used in the successor and limit steps of the recursive construction of the spaces $R_{\alpha}^{p}$ as subspaces of $L_p=L_p[0,1]$ belong to a more general class of operators, which we call $\theta$-compressions. See the comment preceding \cite[Definition 3.4]{konstantos:motakis:brs:2025} for the terminology associated with this class of operators.

\begin{dfn} Let $0< \theta \leq 1$, and $X,Y$ be subspaces $L_p$. We say that a linear operator $T: X \to Y$ is a $\theta$-compression if for all random variables $f \in X$ and Borel sets $A \in \mathcal{B}(\mathbb{R})$ we have the equation 
\begin{align*}
    | [ Tf \in A ] | = \theta |[ f \in A ]| + (1-\theta) \chi_{A}(0)
\end{align*}
Equivalently, 
\begin{align*}
\mathrm{dist}(Tf) = \theta \mathrm{dist}(f) + (1 - \theta) \delta_{0}.
\end{align*}
\end{dfn}

We record two basic properties of $\theta$-compressions. First, the class of $\theta$-compressions is closed under composition, with the corresponding parameters multiplying. More precisely, let $T:X\to Y$ be a $\theta$-compression and let $S:Y\to Z$ be an $\eta$-compression, where $0<\theta,\eta\leq 1$ and $X,Y,Z$ are subspaces of $L_p$. Then $ST:X\to Z$ is a $\theta\eta$-compression. We also record the norm of a $\theta$-compression. In particular, every $\theta$-compression $T:X\to Y$ satisfies $\|Tf\|_p=\theta^{1/p}\|f\|_p$, for all $f\in X$.

For the construction of the successor spaces $R_{\alpha}^{p}$, we use a specific class of $\theta$-compressions, namely $\frac12$-compressions with disjointly supported ranges. For the concept of operators with disjointly supported ranges see \cite[Definition 6.4 and Example 6.5]{konstantos:motakis:brs:2025}. More precisely, these are the operators $T_{[0,1/2)}$ and $T_{[1/2,1)}$, which are particular instances of the following operator.

\begin{ntt} \label{example for compression part b}  Let $I = [a,b)$ be a subinterval of $[0,1]$ with $a < b$. We denote by $T_{I}: L_p \to L_p$ the operator defined by 
\begin{align*}
T_{I}(f)(t) =
\begin{cases}
f(\frac{t-a}{b-a})&: t \in I \\
0&: t \notin I
\end{cases} 
.
\end{align*}
Note that $T_I$ is a $(b-a)$-compression. By \cite[Example 6.5]{konstantos:motakis:brs:2025}, if $I$ and $J$ are disjoint intervals of the above form, then the operators $T_I$ and $T_J$ have disjointly supported ranges.
\end{ntt}

For the construction of the limit spaces $R_{\alpha}^{p}$, we use a specific class of $\theta$-compressions, namely distributional embeddings with independent ranges. Recall that a linear operator $T: X \to Y$, where $X,Y$ are subspaces of $L_p$, is a distributional embedding if it preserves the distribution of every element of $X$, that is for every $f \in X$
\begin{align*}
\mathrm{dist}(Tf) = \mathrm{dist}(f).    
\end{align*}
In this case, we say that the space $X$ distributionally embeds into the space $Y$. Note that $1$-compressions are precisely distributional embeddings. We say that a family of operators $(T_a)_{a\in A}$, where $T_a:L_p\to L_p$, has independent ranges if, for every family $(f_a)_{a\in A}$ with $f_a\in L_p$, the corresponding family $(T_af_a)_{a\in A}$ consists of independent random variables.

For each countable limit ordinal, we fix a family of distributional embeddings with independent ranges.

 \begin{ntt} \label{ntt for independent case}  For each countable limit ordinal $\alpha$, we fix a family of independent measure-preserving functions 
 \begin{align*}
 \phi_{\beta}^{\alpha}: ([0,1], \mathcal{B}([0,1]) \to  ([0,1], \mathcal{B}([0,1]),\;\; \beta < \alpha.  
 \end{align*}
 For every $\beta< \alpha$, denote $T_{\beta}^{\alpha}: L_p \to L_p$ the operator defined by
 \begin{align*}
 T_{\beta}^{\alpha} (f) = f \circ \phi_{\beta}^{\alpha}.    
 \end{align*}
 Observe that the collection $(T_{\beta}^{\alpha})_{\beta< \alpha}$ consists of distributional embeddings with independent ranges.
 \end{ntt}
 
We are now in a position to recall the recursive definition of the spaces $R_{\alpha}^{p}$. See also \cite[Definition 3.9]{konstantos:motakis:brs:2025}.

\begin{dfn} \label{equiv def of BRS spaces}
Let $R_{0}^{p} = \langle \chi_{[0,1]} \rangle$. Let $\alpha$ be an ordinal with $0 < \alpha < \omega_1$ and suppose $R_{\beta}^{p}$  has been defined for all $\beta<\alpha$. If $\alpha = \beta+1$, let
\begin{align*}
R_{\alpha}^{p} = \Big[ T_{[0,\frac{1}{2})} ( R_{\beta}^{p}) \bigcup  T_{ [\frac{1}{2},1) } ( R_{\beta}^{p} ) \Big].
\end{align*}
If $\alpha$ is a limit ordinal, let
\begin{align*}
R_{\alpha}^{p} = \left[ \bigcup_{\beta < \alpha} T_{\beta}^{\alpha}(R_{\beta}^{p})\right].
\end{align*}
\end{dfn}

The above definition is independent of the particular choice of operators made at each step of the recursive construction. See \cite[Remark 3.10]{konstantos:motakis:brs:2025}.

Throughout the remainder of the paper, we primarily work with the spaces $R_{\alpha}^{p,0}$, namely, the subspaces of $R_{\alpha}^{p}$ consisting of all mean-zero random variables.

\begin{ntt}
For every $\alpha < \omega_1$, denote 
\begin{align*}
R_{\alpha}^{p,0} = \lbrace f \in R_{\alpha}^{p}: \mathbb{E}(f) = 0 \rbrace.    
\end{align*}
\end{ntt}

The following remark describes the recursive structure of the spaces $R_{\alpha}^{p,0}$. This structure will be useful throughout the remainder of the paper, as the construction of their unconditional FDDs reflects the recursive nature of the spaces $R_{\alpha}^{p,0}$. See also \cite[Remark 3.12]{konstantos:motakis:brs:2025}.

\begin{rem}  \label{rem for the mean zero BRS spaces } Let $\alpha < \omega_{1}$.  
\begin{enumerate}[label=(\alph*)]  
\item  If $\alpha = n$, then 
\begin{align*}
R_{n}^{p,0} = L_{p}^{n,0}.
\end{align*}

\item  If $\alpha = \beta + n$, then 
\begin{align*}
R_{\alpha}^{p,0} = \Big[ L_{p}^{n,0} \cup \bigcup_{I \in \mathcal{D}^{n}} T_{I} (R_{\beta}^{p,0}) \Big].   
\end{align*}

\item  If $\alpha$ is a limit ordinal, then 
\begin{align*}
R_{\alpha}^{p,0} = \left[ \bigcup_{\beta < \alpha} T_{\beta}^{\alpha}(R_{\beta}^{p,0})\right].
\end{align*}

\end{enumerate}
\end{rem}

The following remark will be used in the proof of the main theorem of this paper, Theorem~\ref{main result}. It states that the distributional image of $R_{\alpha}^{p}$ or $R_{\alpha}^{p,0}$ is complemented in $L_{p}$. For the convenience of the reader, we briefly outline the main steps of the proof and provide the appropriate references.

\begin{rem} \label{crucial remark for orth compl}
Let $\alpha < \omega_1$. In \cite[Theorem 3.3]{konstantos:motakis:brs:2025}, we proved that the space $R_{\alpha}^{p}$ is orthogonally complemented in $L_p$, that is, its complementability in $L_p$ is witnessed by an orthogonal projection that is bounded with respect to $\|\cdot\|_p$. In \cite[Remark 3.14]{konstantos:motakis:brs:2025}, we observed that the same conclusion holds for the space $R_{\alpha}^{p,0}$. Furthermore, \cite[Proposition 2.12(b)]{konstantos:motakis:brs:2025} shows that orthogonal complementability is preserved under distributional isomorphisms. Consequently, if $T \colon R_{\alpha}^{p,0} \to L_p$ is a distributional embedding, then the subspace $T(R_{\alpha}^{p,0})$ is orthogonally complemented in $L_p$. In the proof of Theorem~\ref{main result}, however, we only require that $T(R_{\alpha}^{p,0})$ be complemented in $L_p$.
\end{rem}

\section{Schauder decompositions of $R_{\alpha}^{p,0}$} \label{Section for S-decompostion}

In this section, we recall the construction of the unconditional FDDs for the spaces $R_{\alpha}^{p,0}$ from \cite{konstantos:motakis:brs:2025}, which will be used throughout the paper. For a detailed explanation of this construction, together with illustrative examples, we refer the reader to \cite[Section 4]{konstantos:motakis:brs:2025}.

In \cite{konstantos:motakis:brs:2025}, the authors constructed, for every $\alpha < \omega_1$, an unconditional FDD for the space $R_{\alpha}^{p,0}$. More precisely, for every ordinal $\alpha < \omega_{1}$, there exists a countable well-founded tree $\mathcal{T}_{\alpha}$. For each $\lambda \in \mathcal{T}_{\alpha}$, there exist a positive integer $\kappa(\lambda)$, a parameter $\theta_{\lambda} \in (0,1]$, and a corresponding $\theta_{\lambda}$-compression $T_{\lambda}$, which in turn define the space
$X_{\lambda} = T_{\lambda}(L_{p}^{\kappa(\lambda),0})$. The collection $(X_\lambda)_{\lambda \in \mathcal{T}_{\alpha}}$ forms an unconditional FDD of $R_{\alpha}^{p,0}$.

\subsection{The trees $\mathcal{T}_\alpha$}

For every $1 \leq \alpha < \omega_1$, there exists a countable well-founded tree $\mathcal{T}_\alpha$ that indexes the unconditional FDD of the space $R_{\alpha}^{p,0}$. 

The construction of these trees proceeds recursively along the countable ordinals. To recall this construction, we will use the following notions: the integer part of an ordinal and the concatenation of two finite sequences.

\begin{dfn} For every ordinal $\alpha$ we define its integer part $\kappa (\alpha) \in \mathbb{N}_0$ as follows: Let $\kappa(0) = 0$. If $\alpha$ is a limit ordinal, let 
\begin{align*}
\kappa(\alpha) = 0.
\end{align*}
If $\alpha = \beta +1$, let 
\begin{align*}
\kappa(\alpha) = \kappa(\beta) +1.
\end{align*}
For example, if $n$ is a positive integer, then $\kappa (n) = n$, $\kappa(\omega) =0$ and $\kappa(\omega + n) = n$. When we write $\alpha=\beta+\kappa(\alpha)$, $\beta$ is the limit part of $\alpha$, and $\kappa(\alpha)$ is the integer part of $\alpha$.
\end{dfn}

\begin{dfn}
Let $\lambda=(x_1,\ldots,x_n)$ and $\mu=(y_1,\ldots,y_m)$ be finite sequences. Their concatenation is denoted by $\lambda^\frown\mu$ and is defined by
\[
\lambda^\frown\mu=(x_1,\ldots,x_n,y_1,\ldots,y_m).
\]
\end{dfn}

We are now ready to recall the construction of the trees $\mathcal{T}_{\alpha}$. See also \cite[Definition 4.3]{konstantos:motakis:brs:2025}. An equivalent definition, which provides an explicit description of the elements of $\mathcal{T}_{\alpha}$, is given in \cite[Definition 4.6 and Remark 4.7]{konstantos:motakis:brs:2025}.

\begin{dfn} \label{definition of the indexed trees} For $n \in \mathbb{N}$, let $\mathcal{T}_{n} = \lbrace (n) \rbrace$. Let $\alpha$ be an ordinal with $\alpha < \omega_1$ and suppose $\mathcal{T}_{\beta}$ has been defined for all $\beta<\alpha$. If $\alpha = \beta + \kappa(\alpha)$, let
\begin{align*}
\mathcal{T}_{\alpha} = \lbrace (\alpha) \rbrace \cup \lbrace (\alpha, I)^{\frown} \mu: I \in \mathcal{D}^{\kappa(\alpha)}, \; \mu \in \mathcal{T}_{\beta} \rbrace.    
\end{align*}
If $\alpha$ is a limit ordinal, let
\begin{align*}
\mathcal{T}_{\alpha} = \lbrace (\alpha)^{\frown} \mu: \mu \in \mathcal{T}_{\beta},\; \beta < \alpha \rbrace.    
\end{align*}
\end{dfn}    

Note that the tree associated with a limit ordinal is the disjoint union of the previously constructed trees.

\subsection{Operators and spaces indexed over $\mathcal{T}_\alpha$}

Following the recursive definition of the trees $\mathcal{T}_\alpha$, we recall the structure of the components $X_\lambda$ of the FDDs of the spaces $R_{\alpha}^{p,0}$ associated with the elements $\lambda \in \mathcal{T}_\alpha$.

Indeed, recursively on $1 \leq \alpha < \omega_1$, for every $\lambda \in \mathcal{T}_{\alpha}$, we define a positive integer $\kappa(\lambda)$, a number $0< \theta_{\lambda} \leq 1$ and a $\theta_\lambda$-compression $T_{\lambda} \colon L_{p} \to L_{p}$ as follows. See also \cite[Definition 4.10]{konstantos:motakis:brs:2025}.

\begin{dfn} \label{definition of T_lampda} For $n \in \mathbb{N}$, if $\lambda = (n) \in \mathcal{T}_{n}$, let 
\begin{align*}
\kappa (\lambda) = n, \;\; \theta_{\lambda} = 1,\;\; \text{and}\;\; T_{\lambda} = \mathrm{id}.
\end{align*}
Let $\alpha$ be an ordinal with $\alpha < \omega_1$ and suppose that for every $\lambda \in \mathcal{T}_{\beta}$, $\kappa(\lambda)$, $0< \theta_{\lambda} \leq 1$ and a $\theta_\lambda$-compression $T_{\lambda} \colon L_{p} \to L_{p}$ have been defined for all $\beta<\alpha$.
If $\alpha = \beta + \kappa(\alpha)$ is an infinite successor ordinal and $\lambda \in \mathcal{T}_{\alpha}$, then
\begin{enumerate}[label=(\alph*)]  
\item  if $\lambda = (\alpha)$, let 
\begin{align*}
\kappa (\lambda) = \kappa (\alpha),\;\; \theta_{\lambda} = 1,\;\; \text{and}\;\; T_{\lambda} = \mathrm{id},
\end{align*}

\item if $\lambda = (\alpha, I)^{\frown} \mu$, where $I \in \mathcal{D}^{\kappa(\alpha)}$ and $\mu \in \mathcal{T}_{\beta}$, let 
\begin{align*}
\kappa (\lambda) = \kappa (\mu),\;\; \theta_{\lambda} = \vert I \vert \theta_{\mu}\;\; \text{and}\;\; T_{\lambda} = T_{I} T_{\mu}.
\end{align*}
\end{enumerate}
Let $\alpha$ be a limit ordinal. If $\lambda  \in \mathcal{T}_{\alpha}$, then $\lambda = (\alpha)^{\frown} \mu$, where $\mu\in\mathcal{T}_\beta$ for some $1\leq \beta<\alpha$. Let 
\begin{align*}
\kappa(\lambda) = \kappa (\mu),\;\; \theta_{\lambda} = \theta_{\mu}\;\;  \text{and}\;\; T_{\lambda} = T_{\beta}^{\alpha} T_{\mu}. 
\end{align*}
\end{dfn}

We are now prepared to recall the components of the FDDs. Each component is a compressed copy of a finite dimensional $L_p$ space.

\begin{dfn} \label{def of Xlambda}
Let $\alpha < \omega_{1}$. For every $\lambda \in \mathcal{T}_{\alpha}$, define
\begin{align*}
X_{\lambda} = T_{\lambda}(L_p^{\kappa(\lambda),0}). 
\end{align*}
\end{dfn}

Having defined the above quantities recursively, we are now ready to describe the recursive structure of the components of the FDDs of the spaces $R_{\alpha}^{p,0}$. The following remark describes the relation between the components of the FDD of $R_{\alpha}^{p,0}$ and those associated with the preceding ordinals, which will be used throughout the remainder of the paper. See  also \cite[Remark 4.15]{konstantos:motakis:brs:2025}.

\begin{rem} \label{recursive structure of Xlambda} Let $1 \leq \alpha < \omega_1$.
\begin{enumerate}[label=(\alph*)]
\item If $\alpha = n$, then 
\begin{align*}
\theta_{(n)} = 1,\;\; T_{(n)} = \mathrm{id},\;\; \text{and}\;\; X_{(n)} = L_{p}^{n,0}.
\end{align*}

\item If $\alpha = \beta + \kappa(\alpha)$ is an infinite successor ordinal and $\lambda \in \mathcal{T}_{\alpha}$, then 
\begin{enumerate}[label=(\roman*)]
\item if $\lambda = (\alpha)$, then 
\begin{align*}
\theta_\lambda=1,\;\; T_\lambda  = \mathrm{id}\;\; \text{and}\;\; X_{\lambda} = L_{p}^{\kappa(\lambda),0},
\end{align*}
\item if $\lambda = (\alpha, I)^{\frown}\mu$, where $I \in \mathcal{D}^{\kappa(\alpha)}$ and $\mu \in \mathcal{T}_{\beta}$, then
\begin{align*}
\theta_\lambda = |I|\theta_\mu,\;\; T_{\lambda} = T_{I} T_{\mu}\;\; \text{and}\;\; X_{\lambda} = T_{I}(X_{\mu}).
\end{align*}
\end{enumerate}

\item If $\alpha$ is a limit ordinal and $\lambda \in \mathcal{T}_{\alpha}$, then $\lambda = (\alpha)^{\frown}\mu$, where $\mu \in \mathcal{T}_{\beta}$ for some $\beta < \alpha$. Then 
\begin{align*}
\theta_\lambda = \theta_\mu,\;\; T_{\lambda} = T_{\beta}^{\alpha} T_{\mu}\;\; \text{and}\;\; X_{\lambda} = T_{\beta}^{\alpha}(X_{\mu}).
\end{align*}

\end{enumerate}
\end{rem}

\subsection{The unconditional finite-dimensional decompositions of $R_\alpha^{p,0}$}

In this subsection, we recall the unconditional FDDs of the spaces $R_{\alpha}^{p,0}$ and provide the relevant results from \cite{konstantos:motakis:brs:2025} for completeness.

The following theorem can be obtained immediately from \cite[Proposition 4.16]{konstantos:motakis:brs:2025} and \cite[Theorem 4.17]{konstantos:motakis:brs:2025}.

\begin{thm} Let $\alpha< \omega_1$. The collection $(X_\lambda)_{\lambda \in \mathcal{T}_{\alpha}}$ forms an unconditional FDD of the space $R_{\alpha}^{p,0}$. Thus, 
\begin{align*}
R_{\alpha}^{p,0} = \Big[ \bigcup_{\lambda \in \mathcal{T}_{\alpha}} X_{\lambda}  \Big].    
\end{align*}
    
\end{thm}

\section{The SHAI property of the $R_{\alpha}^{p}$ spaces} \label{section for SHAI of BRS}

The goal of this section is to prove that the spaces $R_{\alpha}^{p,0}$ have the SHAI property. The spaces $R_{n}^{p,0}$ (and $R_{n}^{p}$) have this property since they are finite dimensional. We therefore focus on the infinite-dimensional case. First, we show that for every $\omega \leq \alpha < \omega_1$ the space \(R_{\alpha}^{p,0}\) with its unconditional FDD has property \((\#^{\prime})\), and since it satisfies a finite disjoint lower estimate (see the comment following Remark~\ref{cotype and dis lower estimate}), Theorem \ref{suf cond of SHAI property, var} implies that \(R_{\alpha}^{p,0}\) has the SHAI property. Then the SHAI property of the infinite-dimensional space $R_{\alpha}^{p}$ is derived from the isomorphism $R_{\alpha}^{p} \cong R_{\alpha}^{p,0}$, as it is well-known, all infinite-dimensional complemented subspaces of $L_p$ are isomorphic to their hyperplanes.

In Theorem \ref{main result}, we prove that the spaces $R_{\alpha}^{p,0}$ satisfy property $(\#^{\prime})$. More precisely, we show, by induction on the countable ordinals, that for every $\omega \leq \alpha < \omega_{1}$, there exists an almost disjoint continuum $\lbrace M_\gamma(\alpha): \gamma < \mathfrak{c} \rbrace$ of infinite subsets of $\mathcal{T}_{\alpha}$ such that, for every $\gamma < \mathfrak{c}$, the space $R_{\alpha}^{p,0}$ distributionally embeds into
\begin{align*}
\Big[ X_{\lambda}: \lambda \in M_\gamma(\alpha) \Big],
\end{align*}
and thus, from Remark \ref{crucial remark for orth compl}, this distributional image of $R_{\alpha}^{p,0}$ is complemented in the aforementioned closed linear span. The almost disjoint continuum $\lbrace M_\gamma(\alpha): \gamma < \mathfrak{c} \rbrace$ is constructed recursively. To establish the distributional embedding of $R_{\alpha}^{p,0}$ into
\begin{align*}
\Big[ X_{\lambda}: \lambda \in M_\gamma(\alpha) \Big],
\end{align*}
we employ two different types of distributional embeddings: one for the initial case $\alpha=\omega$ and the general limit ordinal case, and another for the successor ordinal case.

We first describe the distributional embedding that underlies the inductive proof of Theorem \ref{main result} and is used to establish the base case and the limit ordinal case of the induction. The following is a standard result proved using the characteristic function of a random variable, which characterizes its distribution. 

\begin{prop} \label{distributional embedding for the limit case}
Let $(X_n)_{n=1}^\infty$ be a sequence of independent subspaces of $L_p$ and $(T_n:X_n\to L_p)_{n=1}^\infty$, be a sequence of distributional embeddings with independent ranges, that is the spaces $(T_n (X_{n}))_{n=1}^\infty$ are independent. Then, there exists a distributional embedding
\[T:[(X_n)_{n=1}^\infty]\to L_p\]
such that, for all $n\in  \mathbb{N}$,
\[T|_{X_n} = T_n.\]
\end{prop}

We now describe the distributional embedding that underlies the inductive proof of Theorem \ref{main result} and is used to establish the successor ordinal case of the induction.

The proof of the aforementioned distributional embedding is based on characteristic functions of random variables. For this purpose, we establish the following remark.

\begin{rem} \label{easy remark} Let $0 \leq a < b \leq 1$ and $I = [a,b)$. For every bounded measurable function $\phi : \mathbb{R} \to \mathbb{R}$ and $f \in L_{p}$ we have, using a change of variables, that
\begin{align*}
\int_{I} \phi( T_{I}(f) (x)) \mathrm{d}x = \vert I \vert \int_{0}^{1} \phi( f(x)) \mathrm{d}x.
\end{align*}
    
\end{rem}

\begin{prop} \label{distributional embedding for the successor case}
Let $n \in \mathbb{N}$. For every $I \in \mathcal{D}^{n}$, let $X_{I}$ and $Y_{I}$ be subspaces of $L_{p}$ such that $X_{I}$ distributional embeds into $Y_{I}$. Then the space 
\begin{align*}
 \Big[ L_{p}^{n,0} \bigcup \bigcup_{I \in \mathcal{D}^{n}} T_{I} (X_{I})  \Big]   
\end{align*}
distributional embeds into 
\begin{align*}
 \Big[ L_{p}^{n,0} \bigcup \bigcup_{I \in \mathcal{D}^{n}} T_{I} (Y_{I})  \Big].   
\end{align*}
\end{prop}
\begin{proof}
For every $I \in \mathcal{D}^{n}$, denote $J_{I}: X_{I} \to Y_{I}$ a distributional embedding. We define a distributional embedding 
\begin{align*}
T: \Big[ L_{p}^{n,0} \bigcup \bigcup_{I \in \mathcal{D}^{n}} T_{I} (X_{I})  \Big]  \to \Big[ L_{p}^{n,0} \bigcup \bigcup_{I \in \mathcal{D}^{n}} T_{I} (Y_{I})  \Big]   
\end{align*}
as follows. Let $f \in \Big< L_{p}^{n,0} \bigcup \bigcup_{I \in \mathcal{D}^{n}} T_{I} (X_{I})  \Big>$ be arbitrary. We have $f = f_{0} + \sum_{I \in \mathcal{D}^{n}} T_{I}(f_{I})$, where $f_{0} \in L_{p}^{n,0}$ and $f_{I} \in X_{I}$, $I \in \mathcal{D}^{n}$. Define 
\begin{align*}
T(f) =  f_{0} + \sum_{I \in \mathcal{D}^{n}} T_{I}( J_{I} f_{I}).
\end{align*}
It is enough to show that $\mathrm{dist}(T(f)) = \mathrm{dist}(f)$. Indeed, we prove that $\Phi_{T(f)}(t) = \Phi_{f}(t)$, for all $t \in \mathbb{R}$. For every $I \in \mathcal{D}^{n}$, note that $f_{0}|_{I} = a_{I}$ is a constant. From the disjointness of the supports of the functions $T_{I}(J_I f_I)$, $I \in \mathcal{D}^{n}$, and  Remark \ref{easy remark}, we obtain 
\begin{align*}
\Phi_{T(f)}(t) &=  \int_{0}^{1} e^{it \Big(f_{0} + \sum_{I \in \mathcal{D}^{n}} T_{I}( J_{I} f_{I}) \Big)(x)  } \mathrm{d}x  = \sum_{I \in \mathcal{D}^{n}} \int_{I} e^{it \Big(a_{I} +  T_{I}( J_{I} f_{I}) \Big) (x) } \mathrm{d}x \\
& = \sum_{I \in \mathcal{D}^{n}}  \vert I \vert e^{it a_{I}}  \int_{0}^{1} e^{it \Big( J_{I} f_{I} \Big) (x) } \mathrm{d}x = \sum_{I \in \mathcal{D}^{n}}  \vert I \vert e^{it a_{I}}  \int_{0}^{1} e^{it (f_{I} ) (x) } \mathrm{d}x.
\end{align*} 
An analogous argument shows that 
\begin{align*}
\Phi_{f}(t)  = \sum_{I \in \mathcal{D}^{n}}  \vert I \vert e^{it a_{I}}  \int_{0}^{1} e^{it (f_{I} ) (x) } \mathrm{d}x,
\end{align*}
which completes the proof.
\end{proof}

Before proceeding to the proof of the main result of this paper, Theorem \ref{main result}, we require the following proposition, which will be used in the limit case of the inductive argument. It shows that a limit space $R_{\alpha}^{p,0}$ distributionally embeds into an appropriate subspace of itself.

\begin{prop} \label{embed of limit BRS space} \phantom{A}
 \begin{enumerate}[label=(\alph*)]
\item \label{embed of limit BRS space,1}  For every $\alpha<\omega_1$, the space $R_{\beta}^{p,0}$ distributionally embeds into $R_{\alpha}^{p,0}$ for every $\beta \leq \alpha$.

\item \label{embed of limit BRS space,2} Let $\alpha < \omega_{1}$ be a limit ordinal number and let $\Gamma$ be an infinite subset of $\mathbb{N}$. Assume that $(\beta_{n})_{n \in \Gamma}$ is an increasing cofinal sequence in $\alpha$, that is $\beta_{n} < \alpha$, for all $n \in \Gamma$ and $\sup_{n \in \Gamma} \beta_n = \alpha$. Then the space $R_{\alpha}^{p,0}$ distributionally embeds into 
\begin{align*}
\Big[\bigcup_{n \in \Gamma} T_{\beta_n}^{\alpha} (R_{\beta_n}^{p,0})  \Big].    
\end{align*}
 \end{enumerate}
\end{prop}
\begin{proof}  Let us prove the first assertion. We proceed by transfinite induction on $\alpha < \omega_{1}$. Fix $\alpha < \omega_1$. The base case is trivial. Suppose that $\alpha$ is a limit ordinal and $\beta < \alpha$. Then $T_{\beta}^{\alpha}: R_{\beta}^{p,0} \to R_{\alpha}^{p,0}$ is a distributional embedding (the codomain of $T_\beta^{\alpha}$ is $R_{\alpha}^{p,0}$ by Definition \ref{equiv def of BRS spaces}), which proves the limit case. Suppose that $\alpha = \gamma +1$ and let $\beta < \gamma +1$. We prove that $R_{\beta}^{p,0}$ distributionally embeds into $R_{\gamma + 1}^{p,0}$. By the induction hypothesis, $R_{\beta}^{p,0}$ distributionally embeds into $R_{\gamma}^{p,0}$. Let $J: R_{\beta}^{p,0} \to R_{\gamma}^{p,0}$ be a distributional embedding. We define a distributional embedding $T : R_{\beta}^{p,0} \to R_{\gamma+1}^{p,0}$ as follows. Let $f \in R_{\beta}^{p,0}$. Define
\begin{align*}
T(f) = (T_{[0,1/2)} + T_{[1/2,1)})(J(f)),
\end{align*}
(the codomain of $T$ is $R_{\gamma+1}^{p,0}$ by Definition \ref{equiv def of BRS spaces}). By \cite[Proposition 6.7]{konstantos:motakis:brs:2025}, the operator $T_{[0,1/2)} + T_{[1/2,1)}: R_{\gamma}^{p,0} \to R_{\gamma+1}^{p,0}$ is a distributional embedding. Thus, $T$ is a distributional embedding as a composition of two distributional embeddings. 

Now, we prove the second assertion. Choose a sequence $n(\beta)_{\beta < \alpha}$ in $\Gamma$ of pairwise distinct numbers such that $\beta < \beta_{n(\beta)}$. From \ref{embed of limit BRS space,1}, we have that for every $\beta < \alpha$, the space $R_{\beta}^{p,0}$ distributionally embeds into $R_{\beta_{n(\beta)}}^{p,0}$. Hence, for every $\beta < \alpha$, the space  $T_{\beta}^{\alpha} (R_{\beta}^{p,0})$ distributionally embeds into $T_{\beta_{n(\beta)}}^{\alpha} (R_{\beta_{n(\beta)}}^{p,0})$. For every $\beta< \alpha$, denote $J_{\beta}:  T_{\beta}^{\alpha} (R_{\beta}^{p,0}) \to T_{\beta_{n(\beta)}}^{\alpha} (R_{\beta_{n(\beta)}}^{p,0})$ such a distributional embedding. The spaces $(T_{\beta}^{\alpha} (R_{\beta}^{p,0}))_{\beta < \alpha}$ are independent, and the operators $(J_{\beta})_{\beta < \alpha}$ have independent ranges. Consequently, the second assertion follows from Proposition \ref{distributional embedding for the limit case}, which completes the proof.
\end{proof}

We now proceed to the proof of the main result of this work.

\begin{thm} \label{main result} Let $\omega \leq \alpha < \omega_{1}$. Then $R_{\alpha}^{p,0}$ has property $(\#^{\prime})$. Consequently, $R_{\alpha}^{p,0}$ has the SHAI property.
\end{thm}
\begin{proof} By induction on $\omega \leq \alpha < \omega_{1}$, we prove that the space $R_{\alpha}^{p,0}$ has property $(\#^{\prime})$. In particular, we construct an almost disjoint continuum $\lbrace M_\gamma (\alpha): \gamma < \mathfrak{c} \rbrace$ of infinite subsets of $\mathcal{T}_{\alpha}$ such that, for every $\gamma < \mathfrak{c}$, the space $R_{\alpha}^{p,0}$ distributionally embeds into
\begin{align*}
\Big[ X_{\lambda}: \lambda \in M_\gamma (\alpha) \Big].
\end{align*}

Fix an almost disjoint continuum $\lbrace N\gamma: \gamma < \mathfrak{c} \rbrace$ of infinite subsets of $\mathbb{N}$. We begin the inductive process with the base case. Let $\gamma < \mathfrak{c}$. Define
\begin{align*}
M_\gamma (\omega) = \lbrace (\omega,n): n \in N_\gamma \rbrace.
\end{align*}
We claim that the family $\lbrace M_\gamma (\omega): \gamma < \mathfrak{c} \rbrace$ is an almost disjoint continuum. Indeed, let $\gamma, \delta < \mathfrak{c}$ with $\gamma \neq \delta$. Observe that
\begin{align*}
M_\gamma (\omega) \cap  M_\delta (\omega) = \lbrace (\omega,n): n \in N_\gamma \cap N_\delta \rbrace,
\end{align*}
which is finite by the almost disjointness of the family $\lbrace N_\gamma: \gamma < \mathfrak{c} \rbrace$. Let $\gamma < \mathfrak{c}$. Next, we show that $R_{\omega}^{p,0}$ distributionally embeds into
\begin{align*}
\Big[ X_{\lambda}: \lambda \in M_\gamma (\omega) \Big].
\end{align*}
Note that
\begin{align*}
\Big[ X_{\lambda}: \lambda \in M_\gamma (\omega) \Big] = \Big[ T_{n}^{\omega}(L_{p}^{n,0}): n \in N_\gamma \Big].
\end{align*}
Choose an increasing sequence $(N(n))_{n=1}^{\infty}$ in $N_\gamma$ such that $N(n) > n$. For every $n \in \mathbb{N}$, let $J_{n}:T_n^\omega (L_p^{n,0}) \to T^{\omega}_{N(n)} (L_p^{N(n),0})$ be a distributional embedding. Such a distributional embedding can be constructed as $J_{n}(T_{n}^{\omega}h_I) = T_{N(n)}^{\omega}h_I$, for all $I \in \mathcal{D}_{n-1}$.  The spaces $(T_n^\omega (L_p^{n,0}))_{n=1}^{\infty}$ are independent, and the distributional embeddings $(J_{n})_{n=1}^{\infty}$ have independent ranges. Proposition \ref{distributional embedding for the limit case} now completes the proof of the base case.

Now, we establish the successor case of the inductive step. Let $\omega < \alpha < \omega_{1}$ be a limit ordinal, and let $n \in \mathbb{N}$. Assume that the space $R_{\alpha}^{p,0}$ satisfies property $(\#^{\prime})$, witnessed by an almost disjoint continuum $\lbrace M_\gamma(\alpha): \gamma < \mathfrak{c} \rbrace$ of infinite subsets of $\mathcal{T}_{\alpha}$. We next show that $R_{\alpha+n}^{p,0}$ satisfies property $(\#^{\prime})$. We construct an almost disjoint continuum $\lbrace M_\gamma(\alpha + n): \gamma < \mathfrak{c} \rbrace$ of infinite subsets of $\mathcal{T}_{\alpha + n}$. Let $\gamma < \mathfrak{c}$. Define
\begin{align*}
M_\gamma (\alpha + n) = \lbrace (\alpha + n) \rbrace \cup \lbrace (\alpha + n, I)^{\frown}\mu: I \in \mathcal{D}^{n}, \mu \in M_{\gamma}(\alpha) \rbrace.
\end{align*}
We prove that the family $\lbrace M_\gamma(\alpha + n): \gamma < \mathfrak{c} \rbrace$ is almost disjoint. Let $\gamma , \delta < \mathfrak{c}$ with $\gamma \neq \delta$. Note that
\begin{align*}
M_\gamma (\alpha + n) \cap M_\delta (\alpha + n) = \lbrace (\alpha + n) \rbrace \cup \lbrace (\alpha + n, I)^{\frown}\mu: I \in \mathcal{D}^{n}, \mu \in M_{\gamma}(\alpha) \cap M_\delta (\alpha) \rbrace,
\end{align*}
which is finite by the almost disjointness of the family $\lbrace M_\gamma(\alpha): \gamma < \mathfrak{c} \rbrace$. Let $\gamma < \mathfrak{c}$. We show that the space $R_{\alpha+n}^{p,0}$ distributionally embeds into
\begin{align*}
\Big[ X_{\lambda}: \lambda \in M_\gamma (\alpha + n) \Big].
\end{align*}
Note that
\begin{align*}
\Big[ X_{\lambda}: \lambda \in M_\gamma (\alpha + n) \Big] =  \Big[ L_{p}^{n,0} \bigcup \bigcup_{I \in \mathcal{D}^{n}} T_{I} \Big( \Big[ X_\mu: \mu \in M_\gamma (\alpha) \Big] \Big) \Big].
\end{align*}
By the induction hypothesis, the space $R_{\alpha}^{p,0}$ distributionally embeds into
\begin{align*}
\Big[ X_{\mu}: \mu \in M_\gamma (\alpha ) \Big].
\end{align*}
Applying Proposition \ref{distributional embedding for the successor case} with $X_{I} = R_{\alpha}^{p,0}$ and $Y_{I} = \Big[ X_{\mu}: \mu \in M_\gamma (\alpha ) \Big]$, $I \in \mathcal{D}^{n}$, we complete the proof of the successor case.

Next, we establish the limit case of the inductive step. Let $\omega < \alpha < \omega_{1}$ be a limit ordinal number. For every ordinal number $\omega \leq \beta < \alpha$, assume that the space $R_{\beta}^{p,0}$ satisfies property $(\#^{\prime})$, witnessed by an almost disjoint continuum $\lbrace M_\gamma (\beta): \gamma < \mathfrak{c} \rbrace$ of infinite subsets of $\mathcal{T}_{\beta}$. We show that the space $R_{\alpha}^{p,0}$ satisfies property $(\#^{\prime})$. Choose an increasing sequence of ordinal numbers $(\beta_{n})_{n=1}^{\infty}$ that is cofinal in $\alpha$. For $\gamma < \mathfrak{c}$, define 
\begin{align*}
M_{\gamma} (\alpha) = \lbrace (\alpha)^{\frown}\mu: \mu \in M_\gamma (\beta_{n}), n \in N_\gamma  \rbrace.    
\end{align*}
We claim that the family $\lbrace M_\gamma (\alpha): \gamma < \mathfrak{c} \rbrace$ is almost disjoint. Indeed, let $\gamma, \delta < \mathfrak{c}$ with $\gamma \neq \delta$. By the comment following Definition \ref{definition of the indexed trees}, $\mathcal{T}_{\alpha}$ is the disjoint union of the trees $\mathcal{T}_{\beta}$, $\beta < \alpha$. Therefore \begin{align*}
M_\gamma (\alpha) \cap M_\delta (\alpha) = \lbrace (\alpha)^{\frown} \mu: \mu \in M_\gamma(\beta_n) \cap M_\delta (\beta_n), n \in N_\gamma \cap N_\delta \rbrace,    
\end{align*}
which is finite by the almost disjointness of the families $\lbrace N_{\gamma}: \gamma < \mathfrak{c} \rbrace$ and $\lbrace M_\gamma (\beta): \gamma < \mathfrak{c} \rbrace$, $ \omega \leq \beta < \alpha$. Let $\gamma < \mathfrak{c}$. We prove that the space $R_{\alpha}^{p,0}$ distributionally embeds into 
\begin{align*}
\Big[ X_{\lambda}: \lambda \in M_\gamma (\alpha) \Big].    
\end{align*}
Note that the sequence $(\beta_n)_{n \in N_\gamma}$ remains increasing and cofinal in $\alpha$. Thus, Proposition \ref{embed of limit BRS space} \ref{embed of limit BRS space,2} implies that the space $R_{\alpha}^{p,0}$ distributionally embeds into 
\begin{align*}
    \Big[  \bigcup_{n \in N_\gamma} T_{\beta_{n}}^{\alpha} ( R_{\beta_n}^{p,0}) \Big],
\end{align*}
which in combination with the claim that the space 
\begin{align*}
    \Big[  \bigcup_{n \in N_\gamma} T_{\beta_{n}}^{\alpha} ( R_{\beta_n}^{p,0}) \Big]
\end{align*}
distributionally embeds into 
\begin{align*}
\Big[ X_{\lambda}: \lambda \in M_\gamma (\alpha) \Big],  
\end{align*}
completes the proof of the limit case. Now, we prove the aforementioned claim. For every $n \in N_\gamma$, by the induction hypothesis, the space $R_{\beta_n}^{p,0}$ distributionally embeds into   
\begin{align*}
\Big[ X_{\mu}: \mu \in M_\gamma (\beta_n) \Big].
\end{align*}
For every $n \in N_\gamma$, denote 
\begin{align*}
F_{\beta_n} = \Big[ X_{\mu}: \mu \in M_\gamma (\beta_n) \Big],
\end{align*}
and note that 
\begin{align*}
\Big[ X_{\lambda}: \lambda \in M_\gamma (\alpha) \Big] = \Big[  T_{\beta_n}^{\alpha} (F_{\beta_n}) : n \in N_\gamma \Big].
\end{align*}
For every $n \in N_\gamma$, note that the space $T_{\beta_n}^{\alpha} ( R_{\beta_n}^{p,0})$ distributionally embeds into  $T_{\beta_n}^{\alpha} ( F_{\beta_n})$, and denote $J_{\beta_n} : T_{\beta_n}^{\alpha} ( R_{\beta_n}^{p,0}) \to T_{\beta_n}^{\alpha} ( F_{\beta_n})$ such a distributional embedding. Observe that the spaces $(T_{\beta_n}^{\alpha} ( R_{\beta_n}^{p,0}))_{n \in N_\gamma}$ are independent, and the distributional embeddings $(J_{\beta_n})_{n \in N_\gamma}$ have independent ranges. Applying Proposition \ref{distributional embedding for the limit case} we complete the proof of the claim.
\end{proof}

Note that the SHAI property is preserved under isomorphisms.

\begin{rem} Let $\omega \leq \alpha < \omega_1$. Note that $R_{\alpha}^{p}$ is isomorphic to each of its hyperplanes, since it is an infinite-dimensional complemented subspace of $L_{p}$. In particular, $R_{\alpha}^{p} \cong R_{\alpha}^{p,0}$. Thus, the spaces $R_{\alpha}^{p}$ have the SHAI property.  
\end{rem}

\section*{Acknowledgement}

The author gratefully acknowledges Pavlos Motakis for his helpful discussions and valuable feedback.

\bibliographystyle{plain}

\bibliography{bibliography}

@book {albiac:kalton:2006,
    AUTHOR = {Albiac, F. and Kalton, N. J.},
     TITLE = {Topics in {B}anach space theory},
    SERIES = {Graduate Texts in Mathematics},
    VOLUME = {233},
 PUBLISHER = {Springer, New York},
      YEAR = {2006},
     PAGES = {xii+373},
      ISBN = {978-0387-28141-4; 0-387-28141-X},
   MRCLASS = {46B20 (46-01)},
  MRNUMBER = {2192298},
MRREVIEWER = {Gilles\ Godefroy},
}

@article {alspach:1999,
    AUTHOR = {Alspach, D. E.},
     TITLE = {Tensor products and independent sums of {$\mathscr L_p$}-spaces,
              {$1<p<\infty$}},
   JOURNAL = {Mem. Amer. Math. Soc.},
  FJOURNAL = {Memoirs of the American Mathematical Society},
    VOLUME = {138},
      YEAR = {1999},
    NUMBER = {660},
     PAGES = {viii+77},
      ISSN = {0065-9266,1947-6221},
   MRCLASS = {46B20 (46B28 46E30 46M05)},
  MRNUMBER = {1469150},
MRREVIEWER = {Dirk\ Werner},
       DOI = {10.1090/memo/0660},
       URL = {https://doi.org/10.1090/memo/0660},
}

@article {bourgain:rosenthal:schechtman:1981,
    AUTHOR = {Bourgain, J. and Rosenthal, H. P. and Schechtman, G.},
     TITLE = {An ordinal {$L^{p}$}-index for {B}anach spaces, with
              application to complemented subspaces of {$L^{p}$}},
   JOURNAL = {Ann. of Math. (2)},
  FJOURNAL = {Annals of Mathematics. Second Series},
    VOLUME = {114},
      YEAR = {1981},
    NUMBER = {2},
     PAGES = {193--228},
      ISSN = {0003-486X},
   MRCLASS = {46B25 (46E30)},
  MRNUMBER = {632839},
MRREVIEWER = {Gilles Pisier},
       DOI = {10.2307/1971293},
       URL = {https://doi.org/10.2307/1971293},
}

@article{eidelheit:1940,
author = {Eidelheit, M.},
issn = {0039-3223},
journal = {Studia mathematica},
language = {eng ; jpn},
number = {1},
pages = {97-105},
title = {On isomorphisms of rings of linear operators},
volume = {9},
year = {1940},
}

@article{horvath:2020,
author = {Horvath, B.},
address = {WARSZAWA},
copyright = {Copyright 2020 Elsevier B.V., All rights reserved.},
issn = {0039-3223},
journal = {Studia mathematica},
language = {eng},
number = {3},
pages = {259-282},
publisher = {Polish Acad Sciences},
title = {When are full representations of algebras of operators on {B}anach spaces automatically faithful?},
volume = {253},
year = {2020},
}

@article{horvath:kania:shai:2021,
author = {Horvath, B. and Kania, T.},
address = {UK},
copyright = {The Author(s) 2021. Published by Oxford University Press. All rights reserved. For permissions, please e-mail: journals.permissions@oup.com 2021},
issn = {0033-5606},
journal = {Quarterly journal of mathematics},
language = {eng},
number = {4},
pages = {1167-1189},
publisher = {Oxford University Press},
title = {Surjective Homomorphisms from Algebras of Operators on Long Sequence Spaces are Automatically Injective},
volume = {72},
year = {2021},
}

@article{johnson:phillips:schechtman:2022,
title = {The {SHAI} property for the operators on ${L}^{p}$},
journal = {Journal of Functional Analysis},
volume = {282},
number = {4},
pages = {109333},
year = {2022},
issn = {0022-1236},
doi = {https://doi.org/10.1016/j.jfa.2021.109333},
url = {https://www.sciencedirect.com/science/article/pii/S0022123621004158},
author = {W.B. Johnson and N.C. Phillips and G. Schechtman}
}

@phdthesis{konstantos:2026,
  title={Orthogonal factors of operators on the {R}osenthal spaces and the {B}ourgain-{R}osenthal-{S}chechtman spaces},
  author={Konstantos, K.},
  year={2026},
  school={York University Toronto}
}

@article{konstantos:motakis:2025,
title = {Orthogonal factors of operators on the {R}osenthal {$X_{p,w}$} spaces and the {B}ourgain-{R}osenthal-{S}chechtman {$R_\omega^p$} space},
author = {Konstantos, K. and Motakis, P.},
journal = {Journal of Functional Analysis},
volume = {288},
number = {5},
pages = {110802},
year = {2025},
issn = {0022-1236},
doi = {https://doi.org/10.1016/j.jfa.2024.110802},
url = {https://www.sciencedirect.com/science/article/pii/S0022123624004907},

}

@article{konstantos:motakis:brs:2025,
      title={Coordinate systems and distributional embeddings in {B}ourgain-{R}osenthal-{S}chechtman spaces: a framework for operator reduction}, 
      author={Konstantos, K. and Motakis, P.},
      year={2025},
      eprint={2510.24487},
      archivePrefix={arXiv},
      primaryClass={math.FA},
      url={https://arxiv.org/abs/2510.24487},

}

@book{lindenstrauss:tzafriri:1977,
  AUTHOR =	 {Lindenstrauss, J. and Tzafriri, L.},
  TITLE =	 {Classical {B}anach spaces. {I}},
  NOTE =	 {Sequence spaces, Ergebnisse der Mathematik und ihrer Grenzgebiete, Vol. 92},
  PUBLISHER =	 {Springer-Verlag, Berlin-New York},
  YEAR =	 1977,
  PAGES =	 {xiii+188},
  ISBN =	 {3-540-08072-4},
  MRCLASS =	 {46-02 (46A45 46BXX)},
  MRNUMBER =	 {0500056 (58 \#17766)},
  MRREVIEWER =	 {S. V. Kisljakov},
}

@article {rosenthal:1970:Xp,
    AUTHOR = {Rosenthal, H. P.},
     TITLE = {On the subspaces of {$L\sp{p}$} {$(p>2)$} spanned by sequences
              of independent random variables},
   JOURNAL = {Israel J. Math.},
  FJOURNAL = {Israel Journal of Mathematics},
    VOLUME = {8},
      YEAR = {1970},
     PAGES = {273--303},
      ISSN = {0021-2172},
   MRCLASS = {46.35},
  MRNUMBER = {271721},
MRREVIEWER = {A.\ Pietsch},
       DOI = {10.1007/BF02771562},
       URL = {https://doi.org/10.1007/BF02771562},
}

@article {schechtman:1975,
    AUTHOR = {Schechtman, G.},
     TITLE = {Examples of {${\mathcal L}_{p}$} spaces {$(1<p\not=2<\infty )$}},
   JOURNAL = {Israel J. Math.},
  FJOURNAL = {Israel Journal of Mathematics},
    VOLUME = {22},
      YEAR = {1975},
    NUMBER = {2},
     PAGES = {138--147},
      ISSN = {0021-2172},
   MRCLASS = {46B05},
  MRNUMBER = {390722},
MRREVIEWER = {A.\ S.\ Gleit},
       DOI = {10.1007/BF02760162},
       URL = {https://doi.org/10.1007/BF02760162},
}

\end{document}